\documentclass[10pt,a4paper]{amsart}
\usepackage{amsmath, hyperref, bbold}
\usepackage[utf8]{inputenc}
\usepackage[T1]{fontenc}
\usepackage[english]{babel}
\usepackage[all, poly, arc]{xy}
\usepackage{etoolbox}
\makeatletter
\let\ams@starttoc\@starttoc
\makeatother
\usepackage[parfill]{parskip}
\makeatletter
\let\@starttoc\ams@starttoc
\patchcmd{\@starttoc}{\makeatletter}{\makeatletter\parskip\z@}{}{}
\makeatother

\title[Categorification]{Categorification of the ring of cyclotomic integers for products of two primes}
\author{Djalal Mirmohades}
\email{math@djalal.com}

\newtheorem{theorem}{Theorem}[section]
\newtheorem{proposition}[theorem]{Proposition}
\newtheorem{corollary}[theorem]{Corollary}
\newtheorem{lemma}[theorem]{Lemma}

\theoremstyle{definition}
\newtheorem{definition}[theorem]{Definition}

\newcommand{\Z}{{\mathbb Z}}

\begin{document}
\begin{abstract}
Let $n$ be a product of two distinct prime numbers.
We construct a triangulated monoidal category having a Grothendieck ring isomorphic to the ring of $n$:th cyclotomic integers.
\end{abstract}

\maketitle
\tableofcontents

\section{Introduction}\label{intro}

The approach of categorification, which goes back to the work of Crane \cite{Cr} and Frenkel and Crane \cite{CF}, has in the recent years shown that it can be successfully applied to solve problems in various areas of mathematics, see for example \cite{Kh1,CR,MS}.
By now, many mathematical objects were ``categorified'' in different ways.
However, there are still some interesting objects which notoriously resist any sensible categorification.
One of such objects is the ring of $n$:th cyclotomic integers, whose categorification is only know in the case when $n$ is a power of a prime (or two times a power of prime), see \cite{Kh2,Qi,KQ}.
The interest in categorification of cyclotomic integers stems from the hope to categorify quantum groups 
at roots of unity as explained in \cite{Kh2}, while a categorification of quantum groups
in the generic case is well-known, see \cite{KL,Ro}.
In this paper we propose a categorification of the ring of $n$:th cyclotomic integers in the case when $n$ is the product of two distinct primes.
If $n$ is odd, then the ring of $n$:th cyclotomic integers is isomorphic to the ring of $2n$:th cyclotomic integers, so any categorification of the former ring also works for the latter ring.

Let us now discuss the content of the paper in more details.
A category consists of objects and arrows.
Decategorification of a category is about identifying similar objects and then throwing away the arrows. This produces a set. 
However, we want to do this in such a way that as much structure as possible passes over to the set. The loosely defined reverse process, to construct a category which decategorifies to some given mathematical structure, 
is called {\em categorification}, see \cite{Ma} for details.

As a first example, we consider $\mathbb{k}\!-\!\mathrm{mod}$, the category of finite-dimensional vector spaces over a field $\mathbb{k}$.
Already in its name we have a decategorification in mind, namely that its objects are finite-dimensional; linear algebra tells us that two vector spaces are isomorphic if and only if they have the same dimension, so the decategorification in question identifies isomorphic vector spaces by designating a natural number $\dim(X)$ to each vector space $X$.
In fact, some structure is preserved too; we have $\dim(X \oplus Y) = \dim(X) + \dim(Y)$ and $\dim(X \otimes Y) = \dim(X) \times \dim(Y)$ which means that $(\mathbb{N},+ , \times, 0, 1)$ is a decategorification of $(\mathbb{k}\!-\!\mathrm{mod}, \oplus, \otimes, 0, \mathbb{k})$.

Our next example is the Euler characteristic.
It is a decategorification of the category of topological spaces taking value in the set of integers.
However, Euler characteristic does not always carry enough information about a topological space while the topological space itself on the other hand can be very complicated to work with.
So topologist invented something in-between which eventually developed into a field of its own, known as homological algebra.
At its core there is an algebraic category, called \textit{the homotopy category of chain complexes}, here denoted $H_2\!-\!\underline{\mathrm{zmod}}$.
This category carries strictly more information than the Euler characteristic while at the same time being easier to work with than topological spaces.
In particular, the Euler characteristic factors trough $H_2\!-\!\underline{\mathrm{zmod}}$, this means that $H_2\!-\!\underline{\mathrm{zmod}}$ is a categorification of $\Z$.
Objects in $H_2\!-\!\underline{\mathrm{zmod}}$, called chain complexes, are pairs $(X, d)$ consisting of a $\Z$-graded vector space $X$ together with a morphism $d$, called differential, of degree $1$ that satisfies $d^2 = 0$.

We shall work with generalizations of this category denoted $H_N\!-\!\underline{\mathrm{zmod}}$, consisting of objects called $N$-complexes, which are defined in the same way as chain complexes with the only difference being that the differential satisfies $d^N = 0$ instead of $d^2 = 0$.
These categories do not provide new topological invariants (see \cite{Sp}) but they do provide categorifications for other rings than $\Z$.

The subject of $N$-complexes have attracted some interest since Kapranovs paper \cite{Ka} from 1991, see for example \cite{Ab,BHN,Bi,CSW,D1,D2,D3,DH,DK,Es,He,IKM,KW,M2,Ta,Wa}.
The subject shares many properties with classical homological algebra (i.e.\ when $N = 2$) and one reason it does so can be given in the framework of ``Hopfological Algebra'', developed by Khovanov in \cite{Kh2} and Qi in \cite{Qi}.
Our work relies on the following theorem which serves as foundation for Hopfological algebra.
\begin{theorem}[Khovanov]
Any finite-dimensional Hopf algebra $H$ is Frobenius and the stable category of $H$-modules is triangulated monoidal.
\end{theorem}
But before we continue, we need to give a proper definition of decategorification in various algebraic cases.
All categories we consider here are essentially small.

Given an additive category $\mathcal{C}$, its \textit{split Grothendieck group} is defined as the abelian group generated by isomorphism classes $[X]$ of objects $X \in \mathcal{C}$ subject to relations $[X] + [Z] = [Y]$ for each split short exact sequence $0 \to X \to Y \to Z \to 0$.
In this group, $[X \oplus Z] = [X] + [Z]$ but there are no representatives for additive inverses, that is if $[A] + [B] = 0$ then $[A] = [B] = 0$.

Given a triangulated category $\mathcal{C}$, its Grothendieck group, denoted $K_0(\mathcal{C})$, is defined as the abelian group generated by isomorphism classes $[X]$ of objects $X \in \mathcal{C}$ subject to relations $[Y] - [X] = [Z]$ for each triangle $X \to Y \to Z \to X[1]$.
In this group, additive inverses do have representatives because, for any object $X$, there is a triangle $X \to 0 \to X[1] \to X[1]$.

There is an explanation for this difference in behavior regarding additive inverses; the structure of the split Grothendieck group can be viewed as a decategorification of direct sum $\oplus : \mathcal{C} \times \mathcal{C} \to \mathcal{C}$, which actually produces a monoid operation. We made the monoid into a group by introducing formal inverses. 

A triangulated category on the other hand has a structure of triangulation which decategorifies to a group structure all by itself.
The binary group operation then comes from the ``mapping cone'' axiom which decategorifies into subtraction or division.
The Verdier axiom decategorifies to the following axiom 
$$\forall x, y, z : (z-x)-(y-x) = z-y$$ 
and the degenerate triangles to identity axioms 
$$\forall x: x-x = 0, x-0 = x.$$
Any set with an element $0$ and binary operation ``$-$'' that satisfies the axioms above is then a group defined in terms of subtraction.
There is in fact a surprising connection between $N$-complexes and triangulation; it turns out that the simplicial relations coming from $N$-complexes, for all natural numbers $N$, categorifies part of a triangulated structure on the homotopy category which, in turn, categorifies group structure (as described above).
This puts $N$-complexes in a very different context than being mere generalizations of chains complexes, in that context, they provide intrinsic structure on the category of chain complexes which descends to triangulation in the homotopy category.
We refer the reader to \cite{M1} for details.

A triangulated monoidal category $\mathcal{C}$ is a triangulated category having a monoidal structure which is exact in both arguments.
Its Grothendieck ring, also denoted $K_0(\mathcal{C})$, is defined in the same way as the Grothendieck group, but with the additional relation that $[X \otimes Y] = [X] [Y]$.

The approach proposed in this paper generalizes to arbitrary $n$, however, this requires more technical 
work which will be done in a subsequent paper.

\section{\texorpdfstring{$N$}{N}-complexes}\label{n-complexes}

Let $n$ be a positive integer and $\mathbb{k}$ a field of characteristic zero, or positive characteristic which does not divide $n$, containing a primitive $n$:th root of unity $q$.
That is $q \in \mathbb{k}$ satisfies $\Phi_n(q) = 0$ where $\Phi_n$ is the $n$:th cyclotomic polynomial.
The \textit{Taft Hopf algebra}, denoted $H_n$, is given by the associative non-commutative algebra
\begin{equation}\label{H}
H_n = \mathbb{k}\left<k, d\right>\big/(k^n - 1, d^n, dk - qkd)
\end{equation}
with comultiplication $\Delta(k) = k \otimes k$, $\Delta(d) = d \otimes 1 + k \otimes d$, counit $\epsilon(k) = 1$, $\epsilon(d) = 0$ and antipode $S(k) = k^{n-1} = k^{-1}$, $S(d) = -q^{-1}d$.

Let $G_n$ denote the Hopf subalgebra of $H_n$ generated by the group-like elements.
Then $G_n$ can be described as
\begin{equation}\label{G}
G_n = \mathbb{k}[k]\big/(k^n - 1)
\end{equation}
with comultiplication $\Delta(k) = k \otimes k$, counit $\epsilon(k) = 1$ and antipode $S(k) = k^{n-1} = k^{-1}$.
The Hopf algebra $G_n$ is also known as the group Hopf algebra of the cyclic group with $n$ elements over $\mathbb{k}$.

Let $H_n\!-\!\mathrm{mod}$ denote the category of finite-dimensional left $H_n$-modules.
Objects $M$ of $H_n\!-\!\mathrm{mod}$ can be regarded as left $G_n$-modules.
This means that the polynomial $k^n - 1 = \prod_{i=0}^{n-1}(k - q^i)$ acts as zero on $M$.
Thus $M$ decomposes into a direct sum of eigenspaces $M_i$ corresponding to eigenvalues $q^i$.
Let the \textit{implicit grading} on $M$ denote the $\Z\big/(n)$-grading so that $M_i$ has degree $i$ modulo $n$.
By construction $k$ has implicit degree zero. The relation $dk = qkd$ implies that $d$ has implicit degree $-1$ modulo $n$.

Now regard $H_n$ as a graded Hopf algebra with $k$ having degree $0$ and $d$ degree $1$.
Let $H_n\!-\!\mathrm{gmod}$ denote the category of $\Z$-graded finite-dimensional left $H_n$-modules.
We call this grading the \textit{explicit grading}.

For an object of $H_n\!-\!\mathrm{gmod}$, define the \textit{total grading} as
$$\text{total grading} := \text{explicit grading} + \text{implicit grading} \pmod{n}.$$
The action of $k$ and $d$ has explicit $\times$ implicit degree $(0,\overline{0})$ and $(1,\overline{-1})$ respectively, so both of them has an action of total degree equal to zero modulo $n$.
This means that given an object $M$ of $H_n\!-\!\mathrm{gmod}$, its homogeneous total degree zero part $M_{\overline{0}}$ is a submodule of $M$.
Define the \textit{homological grading} on $M_{\overline{0}}$ as the $\Z$-grading given by the kernel of $+$:
\[\xymatrix@R=1em{
& \text{\small homological} & \text{\small expl.} \times \text{\small impl.} & \text{\small total} & \\
0 \ar[r] & \Z \ar[r]
& \Z \times \Z\big/(n) \ar[r]^(.55)+ & \Z\big/(n) \ar[r] & 0 .\\
& i \ar@{|->}[r] & (i, \overline{-i}) \ar@{|->}[r] & \overline{0} \\
}\]
In other words, a homogeneous element in $M_{\overline{0}}$ which has explicit degree $i$ will have implicit degree $\overline{-i} = -i + (n)$ and we define its homological degree to be $i$.
The implicit grading and hence the action of $k$ can then be inferred from the homological grading; for a homogeneous element $x$ of homological degree $i$, we have $kx = q^{-i}x$.
This means that the action of $k$ on $M_{\overline{0}}$ carries no information and may as well be forgotten without any loss of information.
In other words, the module $M_{\overline{0}}$ can be regarded as a $\Z$-graded $\mathbb{k}[d]\big/(d^7)$-module, with $d$ having degree $1$.
This makes $M_{\overline{0}}$ a $7$-complex (See definition of $N$-complexes in \cite{Ka}).

The following diagram illustrates the explicit $\times$ implicit $\Z \times \Z\big/(7)$-grading of objects in $H_7-\mathrm{gmod}$ as the dotted coordinate system:

\vspace{20mm}

\hspace{11mm}
\begin{xy} 
\xybox{0;<4mm, 0mm>:<0mm,10mm>::
{\xypolygon7"1"{~={0}~>{} \bullet}}, {\ellipse(1){.}}, +(3.3,0),
{\xypolygon7"2"{~={0}~>{} \bullet}}, {\ellipse(1){.}}, +(3.3,0),
{\xypolygon7"3"{~={0}~>{} \bullet}}, {\ellipse(1){.}}, +(3.3,0),
{\xypolygon7"4"{~={0}~>{} \bullet}}, {\ellipse(1){.}}, +(3.3,0),
{\xypolygon7"5"{~={0}~>{} \bullet}}, {\ellipse(1){.}}, +(3.3,0),
{\xypolygon7"6"{~={0}~>{} \bullet}}, {\ellipse(1){.}}, +(3.3,0),
{\xypolygon7"7"{~={0}~>{} \bullet}}, {\ellipse(1){.}}, +(3.3,0),
{\xypolygon7"8"{~={0}~>{} \bullet}}, {\ellipse(1){.}},
"11";"22"**@{-}*\dir{>},"11";"81"**@{.},
"22";"33"**@{-}*\dir{>},"12";"82"**@{.},
"33";"44"**@{-}*\dir{>},"13";"83"**@{.},
"44";"55"**@{-}*\dir{>},"14";"84"**@{.},
"55";"66"**@{-}*\dir{>},"15";"85"**@{.},
"66";"77"**@{-}*\dir{>},"16";"86"**@{.},
"77";"81"**@{-}*\dir{>},"17";"87"**@{.},
}
\end{xy}
\vspace{1ex}

The differential $d \in H_7$ acts in parallel to the spiral. 
Since there are $7$ spirals parallel to the one above, every module $M$ decomposes into a direct sum of $7$ submodules $M_{\overline{0}} \oplus \cdots \oplus M_{\overline{6}}$ indexed by its homogeneous total degree modulo $7$.

Because there are no non-zero morphisms between modules of distinct total degree, the entire category $H_n\!-\!\mathrm{gmod}$ splits into a direct sum 
$$H_n\!-\!\mathrm{gmod} = \bigoplus_{i = 0}^{n-1} H_n\!-\!\mathrm{gmod}_{\overline{i}}$$
where $H_n\!-\!\mathrm{gmod}_{\overline{i}}$ denotes the homogeneous subcategory of total degree $i$ modulo $n$.
The tensor product in $H_n\!-\!\mathrm{gmod}$ is additive in both explicit and implicit grading, this implies additivity in total and homological grading as well.
Thus, we may think of the tensor product inherited from $H_n\!-\!\mathrm{gmod}$ as a bi-additive functor
$$\otimes : H_n\!-\!\mathrm{gmod}_{\overline{i}} \times H_n\!-\!\mathrm{gmod}_{\overline{j}} \longrightarrow H_n\!-\!\mathrm{gmod}_{\overline{i+j}}.$$
This means in particular that the subcategory $H_n\!-\!\mathrm{gmod}_{\overline{0}}$ is endowed with a tensor product.
Since there are no non-zero morphisms between different summands $H_n\!-\!\mathrm{gmod}_{\overline{i}}$, this decomposition carries over to a decomposition $H_n\!-\!\underline{\mathrm{gmod}}_{\overline{i}}$ of the stable category $H_n\!-\!\underline{\mathrm{gmod}}$.
A definition of the stable category can be found in \cite[Section 4--6]{Ke} for the general Frobenius case or \cite[Section 2]{Qi} for the hopfological case.
Let
\begin{align}\label{definition_zmod}
\begin{array}{r@{}l}
H_n\!-\!\mathrm{zmod} &{}:= H_n\!-\!\mathrm{gmod}_{\overline{0}}, \\
H_n\!-\!\underline{\mathrm{zmod}} &{}:= H_n\!-\!\underline{\mathrm{gmod}}_{\overline{0}}.
\end{array}
\end{align}

\begin{proposition}
The category $H_N\!-\!\mathrm{zmod}$ is monoidally isomorphic to the category of finite-dimensional $\Z$-graded $N$-complexes (defined by Kapranov in \cite{Ka}) of vector spaces over the field $\mathbb{k}$.
\end{proposition}

\begin{proof}
Follows from the discussion above.
\end{proof}

Analogously to the above, let $H_n \otimes H_m\!-\!\mathrm{gmod}$ denote the category of $\Z \times \Z$-graded left $H_n \otimes H_m$-modules, where we require $n$ and $m$ to be relatively prime.
The category has two explicit as well as two implicit gradings.
Its total degree lives in $\Z\big/(n) \! \times \! \Z\big/(m)$, hence it induces a $\Z \! \times \! \Z$ homological grading (i.e.\ bigrading) in the part of total degree $(\overline{0}, \overline{0})$.
We denote the full subcategory of $H_n \! \otimes \! H_m \!-\!\mathrm{gmod}$ of total degree $(\overline{0}, \overline{0})$ by $H_n \! \otimes \! H_m \!-\!\mathrm{zmod}$.

Since $H_n \otimes H_m$ is a Hopf algebra, the category $H_n \! \otimes \! H_m \!-\!\mathrm{gmod}$ has a tensor product.
Just as in the case of $H_n\!-\!\mathrm{zmod}$, the category $H_n \! \otimes \! H_m \!-\!\mathrm{zmod}$ inherits a tensor product from $H_n \! \otimes \! H_m \!-\!\mathrm{gmod}$.
This then descends to a exact monoidal structure on the triangulated stable category $H_n \! \otimes \! H_m \!-\!\underline{\mathrm{zmod}}$.

The whole purpose of reformulating $N$-complexes in terms of a subcategory of graded Hopf algebra modules for a finite-dimensional Hopf algebra is to enter the domain of ``Hopfological algebra''.
This framework provides a triangulated and exact monoidal structure on the stable category $H_n \! \otimes \! H_m \!-\!\underline{\mathrm{zmod}}$.
Hopfological algebra was initiated by Khovanov in \cite{Kh2} and Qi in \cite{Qi} where proof of the triangulated monoidal structure of can be found.

It is worth noting that a description of $N$-complexes in terms of Hopf algebra modules was given by Bichon in \cite{Bi} already in 2003 which in turn generalizes a result for the case $N = 2$ by Pareigis in \cite{Pa} from 1981.
However, this description does not fit in the framework of Hopfological algebra because the Hopf algebras in question are infinite-dimensional.

\section{A Sum of Two Ideals}

Tensor the embedding of Hopf algebras $G_n \to H_n$ (described in Section \ref{n-complexes}) with $H_m$ to get a Hopf algebra morphism
$G_n \otimes H_m \to H_n \otimes H_m$.
This induces the exact restriction functor
$$P_0 : H_n \! \otimes \! H_m \!-\!\mathrm{zmod} \longrightarrow G_n \! \otimes \! H_m \!-\!\mathrm{zmod}.$$
The functor $P_0$ can be described in terms of a tensor product
$$P_0 \simeq {}_{(G_n \otimes H_m)}{F(0, 0)}_{(H_n \otimes H_m)} \otimes \_$$
or in terms of a hom-functor
$$P_0 \simeq \operatorname{Hom}_{{(H_n \otimes H_m)}}\left({}_{(H_n \otimes H_m)}{F(0, 0)}_{(G_n \otimes H_m)}, \_ \right)$$
restricted to the subcategory $H_n \! \otimes \! H_m \!-\!\mathrm{zmod}$, where $F(0, 0)$ is the graded free \mbox{$H_n \! \otimes \! H_m$}-module generated by one generator in explicit degree $(0, 0)$.
This description works because tensor and hom with $F(0, 0)$ preserves total degree.
But if a functor preserves total degree, then so must any of its adjoints.
By the usual tensor-hom adjunction, these descriptions then provide an associated right adjoint and a left adjoint to $P_0$, respectively.
The functor $P_0$ descends to an exact functor of triangulated categories
$$\underline{P_0} : H_n \! \otimes \! H_m \!-\!\underline{\mathrm{zmod}} \longrightarrow G_n \! \otimes \! H_m \!-\!\underline{\mathrm{zmod}}$$
because $P_0$ preserve tensor product and projective objects.
Analogously, $P_1$ and $\underline{P_1}$ denotes the exact functors corresponding to the Hopf algebra embedding $H_n \otimes G_m \to H_n \otimes H_m$.
\[\xymatrix@C=0mm{
& H_n \! \otimes \! H_m \!-\!\mathrm{zmod} \ar[dl]_{P_0} \ar[dr]^{P_1} & \\
G_n \! \otimes \! H_m \!-\!\mathrm{zmod} & & H_n \! \otimes \! G_m \!-\!\mathrm{zmod}\\
}\]
We shall give an explicit construction of the parts we need and leave re rest to the reader.
These are the construction of right adjoint 
$$R_0 : G_n \! \otimes \! H_m \!-\!\mathrm{zmod} \longrightarrow H_n \! \otimes \! H_m \!-\!\mathrm{zmod}$$
described (only) on objects, together with the unit of this adjunction.

\begin{definition}
The functor $R_0$ is defined on objects $Y$ by first describing $R_0 Y$ as in terms of a $G_n \! \otimes \! H_m$-module, then defining the action of the remaining generator $d \otimes 1$, which turns it into an $H_n \! \otimes \! H_m$-module.
So let $Y$ be an object in $G_n \! \otimes \! H_m \!-\!\underline{\mathrm{zmod}}$ and define
$$R_0 Y := Y \oplus Y\{1, 0\} \oplus \cdots \oplus Y\{n-2, 0\} \oplus Y\{n-1, 0\}$$
where $Y\{i, j \}$ denotes a shift given by $Y\{i, j \}^{(a,b)} = Y^{(a+i,b+j)}$ in the homological degree, as indicated by upper indices. while keeping the total degree unchanged.
This is achieved by a shift of $(i, j)$ in the explicit degree together with a shift of $(-i, -j)$ modulo $n\Z \times m\Z$ in the implicit degree.
The remaining generator $d_0 = d \otimes 1$ of $H_n \! \otimes \! H_m$ acts on $R_0 Y$ as:
$$d_0 (y_0, y_1, \dots, y_{n-2}, y_{n-1}) := (y_1, y_2, \dots, y_{n-1}, 0).$$
We need to check that the action of $d_0$ satisfies the relations given in \ref{H}. We obviously have $d_0^n = 0$. 
It remains to check that $d_0 k_0 y = q_0 k_0 d_0 y$ where $q_0$ is equal to $q$ of $H_n$, $k_0 = k \otimes 1$ and $y \in R_0 Y$.
Because of the shift in the implicit degree, $k_0$ acts as:
$$k_0 (y_0, y_1, \dots, y_{n-2}, y_{n-1}) := (k_0 y_0, q_0 k_0 y_1, \dots, q_0^{n-2} k_0 y_{n-2}, q_0^{n-1} k_0 y_{n-1})$$
written in terms of non-shifted coordinates.
Hence we have
$$(q_0 k_0 y_1, q_0^2 k_0 y_2, \dots, q_0^{n-1} k_0 y_{n-1}, 0) = q_0 (k_0 y_1, q_0 k_0 y_2, \dots, q_0^{n-2} k_0 y_{n-1}, 0)$$
which shows that $d_0 k_0 y = q_0 k_0 d_0 y$.
Now if $Y$ is a zero object in the stable category $G_n \! \otimes \! H_m \!-\!\underline{\mathrm{zmod}}$, then $Y$ is isomorphic to a direct sum of modules of the form $\mathbb{k} \otimes H_m \{i, j\}$ and each summand is mapped by $R_0$ to a free module, which is a zero object in $H_n \! \otimes \! H_m \!-\!\underline{\mathrm{zmod}}$.
\end{definition}

\begin{definition}
For every object $X$ in the category $H_n \! \otimes \! H_m \!-\!\mathrm{zmod}$, define the monomorphism $\eta_X : X \to R_0 P_0 X$ as:
\[\xymatrix{
x \ar@{|->}[r] & (x, d_0 x, \dots, d_0^{n-2} x, d_0^{n-1} x) .\\
}\]
It is clear that $\eta_X$ preserves the action of $1 \otimes k$ and $1 \otimes d$.
We first check that $\eta_X$ preserves the action of $k_0 = k \otimes 1$:
\[\xymatrix@R=1ex{
x \ar@{|->}[r]^(.3){\eta_X} \ar@{|->}[dddd]_{k_0} & (x, d_0 x, \dots, d_0^{n-2} x, d_0^{n-1} x) \ar@{|->}[ddd]^{k_0} \\ \\ \\
& (k_0 x, q_0 k_0 d_0 x, \dots, q_0^{n-2} k_0 d_0^{n-2} x, q_0^{n-1} k_0 d_0^{n-1} x) \ar@{}[d]|{=} \\
k_0 x \ar@{|->}[r]_(.3){\eta_X} & (k_0 x, d_0 k_0 x, \dots, d_0^{n-2} k_0 x, d_0^{n-1} k_0 x)\\
}\]
which commutes because $d_0 k_0 = q_0 k_0 d_0$.
Next, we check that $\eta_X$ preserves the action of $d_0 = d \otimes 1$:
\[\xymatrix{
x \ar@{|->}[r]^(.25){\eta_X} \ar@{|->}[d]_{d_0} & (x, d_0 x, \dots, d_0^{n-2} x, d_0^{n-1} x) \ar@{|->}[d]^{d_0} \\
d_0 x \ar@{|->}[r]_(.25){\eta_X} & (d_0 x, d_0^2 x, \dots, d_0^{n-1} x, 0)\\
}\]
which commutes because $d_0^n = 0$.
\end{definition}

\begin{lemma}\label{factors}
Let $Y$ be an object in the kernel of the functor 
$$\underline{P_1} : H_n \! \otimes \! H_m \!-\!\underline{\mathrm{zmod}} \longrightarrow H_n \! \otimes \! G_m \!-\!\underline{\mathrm{zmod}}.$$
Then, for any object $X$ and morphism $f : X \to Y$ in $H_n \! \otimes \! H_m \!-\!\mathrm{zmod}$, we have that $f$ factors through $\eta_X$:
\[\xymatrix{
& R_0 P_0 X \ar@{-->}[dr]^\exists & \\
X \ar[ur]^{\eta_X} \ar[rr]_f & & Y .\\
}\]
\end{lemma}
\begin{proof}
Note that the category $H_n \! \otimes \! H_m \!-\!\mathrm{zmod}$ has the same objects as its stable version, so the formulation of the proposition makes sense.
In this proof, we only deal with the first component of the homological grading (the $H_n$-component), which we simply call grading (or degree). As before, $d_0 = d \otimes 1$.

In the category $G_n \! \otimes \! H_m$ we have $P_0 X = \bigoplus_i X^i$, where $X^i$ has homogeneous degree $i$.
Then $R_0 P_0 X = R_0 \left( \bigoplus_i \! X^i \right) = \bigoplus_i R_0 X^i$.
Let $I$ be a section of $P_0$ such that $I M$ is the unique $H_n \! \otimes \! H_m$-module that satisfies $P_0 I M = M$ and $d_0 I M = 0$.
For each $i$, we then have $\eta_{I X^i} I X^i = d_0^{n-1} R_0 P_0 I X^i = d_0^{n-1} R_0 X^i$ as depicted here:
\[\xymatrix{
& & & & X^i \ar[d]^{\eta} \\
X^i \ar[r]^{d_0} & X^i \ar[r]^{d_0} & \cdots \ar[r]^{d_0} & X^i \ar[r]^{d_0} & X^i.
}\]

We shall now construct a morphism $g : R_0 P_0 X \to Y$ that satisfies $g \eta_X = f$, by induction on one summand $g_i : R_0 X^i \to Y$ at a time.

We begin with the highest degree $b$ at which $X^b$ is non-zero.
If $x \in X$ is a homogeneous element of degree $b$, then $d_0 x = 0$, which implies $d_0 f x = f d_0 x = 0$.
But 
$$f x \in \operatorname{ker} d_0 \Longrightarrow f x \in \operatorname{im} d_0^{m-1}$$
because $Y$ lies in the kernel of $\underline{P_1}$.
This means we may choose a morphism $g_b : R_0 X^b \to Y$ such that 
$$g_b \eta_X x = f x$$
for all $x$ of homogeneous degree $b$. This was the basis of the induction.

Assume we have constructed $g_i$ for all $i > j$ in such a way that the sum $s_j = \sum_{i=j+1}^b g_i$ satisfies $s_j \eta_X x = f x$ for every $x$ of homogeneous degree greater than $j$.
It remains to construct a $g_j : R_0 X^j \to Y$ which takes care of degree $j$.

But we already have $f d_0 x = s_j \eta_X d_0 x$ for every $x$ of homogeneous degree $j$ because $d_0 x$ has degree $j+1$.
This implies $0 = (f - s_j \eta_X) d_0 x = d_0 (f - s_j \eta_X) x$.
Just like above, we can then choose a morphism $g_j : R_0 X^j \to Y$ such that 
$$g_j \eta_X x = (f-s_j \eta_X)x$$
for all $x$ of homogeneous degree $j$. That means $f x = (s_j + g_j)\eta_X x$.

This induction terminates because $X$ is finite-dimensional.
\end{proof}

Similarly to Lemma \ref{factors}, if $Y$ is an object in the kernel of the functor 
$$\underline{P_0} : H_n \! \otimes \! H_m \!-\!\underline{\mathrm{zmod}} \longrightarrow G_n \! \otimes \! H_m \!-\!\underline{\mathrm{zmod}},$$
then for any object $X$ and morphism $f : X \to Y$ in $H_n \! \otimes \! H_m \!-\!\mathrm{zmod}$, $f$ factors through $\eta_X$:
\[\xymatrix{
& R_1 P_1 X \ar@{-->}[dr]^\exists & \\
X \ar[ur]^{\eta_X} \ar[rr]_f & & Y ,\\
}\]
where $\eta_X$ is defined in the same way as before.

\begin{proposition}\label{thick_kernel}
The ideals $\ker \underline{P_0}$ and $\ker \underline{P_1}$ of $H_n \! \otimes \! H_m \!-\!\underline{\mathrm{zmod}}$ are thick.
\end{proposition}

\begin{proof}
This follows from the fact that the functors $\underline{P_0}$ and $\underline{P_1}$ are exact (see \cite[Subsection 4.5]{Kr}).
\end{proof}

\begin{theorem}\label{orthogonal}
The ideals $\ker \underline{P_0}$ and $\ker \underline{P_1}$ of $H_n \! \otimes \! H_m \!-\!\underline{\mathrm{zmod}}$ are orthogonal to each other, that is any morphism from an object in one of the subcategories to an object in the other is equal to zero.
\end{theorem}
\begin{proof}
Assume $X$ and $Y$ lie in the kernel of $\underline{P_0}$ and $\underline{P_1}$, respectively.
Then we have isomorphisms $\underline{R_0} \: \underline{P_0} X \simeq \underline{R_0} 0 \simeq 0$ in $H_n \! \otimes \! H_m \!-\!\underline{\mathrm{zmod}}$.
But since any $f : X \to Y$ factors trough $R_0 P_0 X = \underline{R_0} \: \underline{P_0} X$ by Lemma \ref{factors}, we have $f = 0$ in the stable category.
The proof that morphisms $g : Y \to X$ are equal to zero is similar to the above except that $P_0$ and $P_1$ have swapped roles to let $g$ factor trough $R_1 P_1 Y$.
\end{proof}

\begin{corollary}
The intersection $\ker \underline{P_0} \cap \ker \underline{P_1}$ is zero.
\end{corollary}

\begin{proof}
The subcategories are orthogonal by the preceding theorem.
\end{proof}

\begin{corollary}\label{thick_ideal}
The sum of ideals $\ker \underline{P_0} + \ker \underline{P_1}$ of $H_n \! \otimes \! H_m \!-\!\underline{\mathrm{zmod}}$ is thick.
\end{corollary}

\begin{proof}
We need to show that the sum is a full triangulated subcategory, closed under direct summands.
Both of the ideals $\ker \underline{P_0}$ and $\ker \underline{P_1}$ are thick by Proposition \ref{thick_kernel}.
Hence their sum is a full subcategory closed under direct summands.

It remains to check that it is closed under the shift functor and the ``mapping cone'' axiom (that any morphism can be extended to a triangle).
The shift functor is additive which proves closeness, but so is the ``mapping cone'' axiom:
Every morphism $f$ in $\ker \underline{P_0} + \ker \underline{P_1}$ can be written as a direct sum $f_0 \oplus f_1$ where $f_i \in \ker \underline{P_i}$.
Consequently, any triangle beginning with $f$ also decomposes into a direct sum of two triangles, one summand for each subcategory.
\end{proof}

These nice properties of the ideals $\ker \underline{P_0}$ and $\ker \underline{P_1}$ are a consequence of the $\Z \! \times \! \Z$-grading (homological bi-grading) on $H_n \! \otimes \! H_m \!-\!\underline{\mathrm{zmod}}$.
In the category $H_n \! \otimes \! H_m \!-\!\underline{\mathrm{mod}}$, where we only have $\Z\big/(n) \! \times \! \Z\big/(m)$-grading (implicit bi-grading), things are more complicated.
For example we have the following 15-dimensional object in $H_3 \! \otimes \! H_5 \!-\! \underline{\mathrm{mod}}$

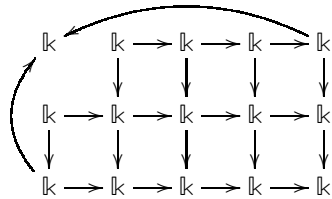
\vspace{-2mm}
\[\xymatrix@R=5mm@C=5mm{
\mathbb{k} & \mathbb{k} \ar[r] \ar[d] & \mathbb{k} \ar[r] \ar[d] & \mathbb{k} \ar[r] \ar[d] & \mathbb{k} \ar@/_5mm/[llll] \ar[d] \\
\mathbb{k} \ar[r] \ar[d] & \mathbb{k} \ar[r] \ar[d] & \mathbb{k} \ar[r] \ar[d] & \mathbb{k} \ar[r] \ar[d] & \mathbb{k} \ar[d] \\
\mathbb{k} \ar[r] \ar@/^5mm/[uu] & \mathbb{k} \ar[r] & \mathbb{k} \ar[r] & \mathbb{k} \ar[r] & \mathbb{k}\\
}\]
where the drawn vertical and horizontal arrows tell where the action of $d_0$ and $d_1$, respectively, is the identity action.
This is a non-zero object in $H_3 \! \otimes \! H_5 \!-\!\underline{\mathrm{mod}}$ because it not isomorphic to (any shift of) the free module $H_3 \! \otimes \! H_5$ in $H_3 \! \otimes \! H_5 \!-\!\mathrm{mod}$.
If we unwrap the above diagram, we get:

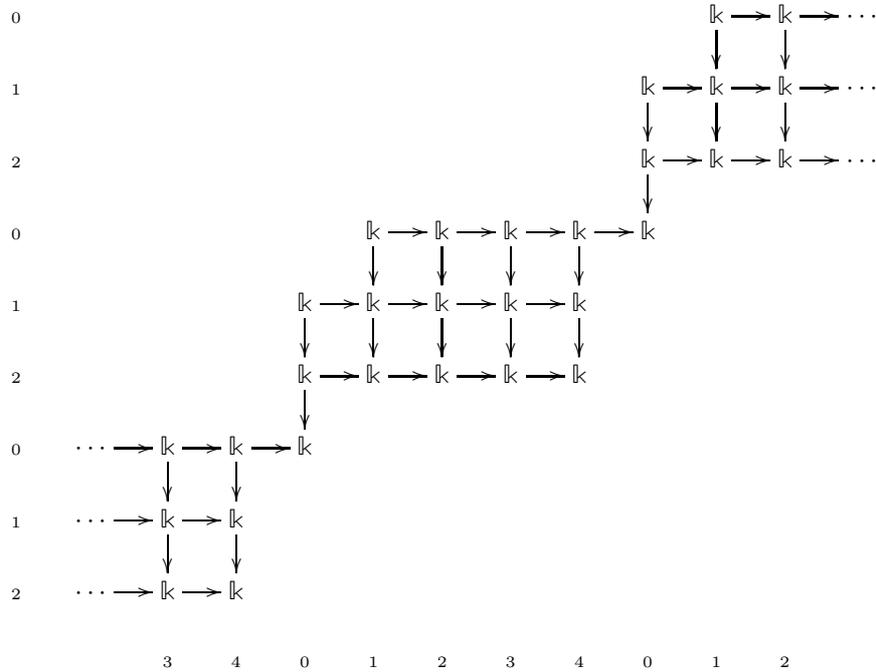
\[\xymatrix@R=5mm@C=5mm{
{\mbox{\tiny 0}} & & & & & & & & & & \mathbb{k} \ar[d] \ar[r] & \mathbb{k} \ar[d] \ar[r] & \cdots \\
{\mbox{\tiny 1}} & & & & & & & & & \mathbb{k} \ar[d] \ar[r] & \mathbb{k} \ar[d] \ar[r] & \mathbb{k} \ar[d] \ar[r] & \cdots \\
{\mbox{\tiny 2}} & & & & & & & & & \mathbb{k} \ar[d] \ar[r] & \mathbb{k} \ar[r] & \mathbb{k} \ar[r] & \cdots \\
{\mbox{\tiny 0}} & & & & & \mathbb{k} \ar[r] \ar[d] & \mathbb{k} \ar[r] \ar[d] & \mathbb{k} \ar[r] \ar[d] & \mathbb{k} \ar[r] \ar[d] & \mathbb{k} \\
{\mbox{\tiny 1}} & & & & \mathbb{k} \ar[r] \ar[d] & \mathbb{k} \ar[r] \ar[d] & \mathbb{k} \ar[r] \ar[d] & \mathbb{k} \ar[r] \ar[d] & \mathbb{k} \ar[d] \\
{\mbox{\tiny 2}} & & & & \mathbb{k} \ar[r] \ar[d] & \mathbb{k} \ar[r] & \mathbb{k} \ar[r] & \mathbb{k} \ar[r] & \mathbb{k} \\
{\mbox{\tiny 0}} & \cdots \ar[r] & \mathbb{k} \ar[r] \ar[d] & \mathbb{k} \ar[r] \ar[d] & \mathbb{k} \\
{\mbox{\tiny 1}} & \cdots \ar[r] & \mathbb{k} \ar[r] \ar[d] & \mathbb{k} \ar[d] \\
{\mbox{\tiny 2}} & \cdots \ar[r] & \mathbb{k} \ar[r] & \mathbb{k} \\
& & {\mbox{\tiny 3}} & {\mbox{\tiny 4}} & {\mbox{\tiny 0}} & {\mbox{\tiny 1}} & {\mbox{\tiny 2}} & {\mbox{\tiny 3}} & {\mbox{\tiny 4}} & {\mbox{\tiny 0}} & {\mbox{\tiny 1}} & {\mbox{\tiny 2}} \\
}\]
where the grading is modulo $3$ vertically and modulo $5$ horizontally.
The unwrapped picture makes it easy to see that the actions of $d_0$ and $d_1$ commute and that this non-zero object lies in both of the ideals $\ker \underline{P_0}$ and $\ker \underline{P_1}$ of $H_3 \! \otimes \! H_5 \!-\!\underline{\mathrm{mod}}$.
This implies that the two ideals are not orthogonal in this setting.
\pagebreak

\section{Grothendieck Groups and Rings}

See Section \ref{intro} for definition and properties of Grothendieck Groups and Rings.

\begin{proposition}\label{grothendieck_quotient}
Let $\mathcal{H}$ denote the category $H\!-\!\mathrm{mod}$, $H\!-\!\mathrm{gmod}$ or $H\!-\!\mathrm{zmod}$ where $H$ is a finite-dimensional, in the latter cases graded, Hopf algebra.
Let $I$ be the subgroup of the Grothendieck ring $K_0(\mathcal{H})$ generated by the classes of projective objects of $\mathcal{H}$.
Then $I$ is an ideal in $K_0(\mathcal{H})$ and the Grothendieck ring of the stable category $\underline{\mathcal{H}}$ can be described as follows:
$$K_0(\underline{\mathcal{H}}) \simeq K_0(\mathcal{H})\big/ I.$$
\end{proposition}

\begin{proof}
In this proof, all drawn diagrams commute.
Let $G$ and $\underline{G}$ denote the split Grothendieck group of $\mathcal{H}$ and $\underline{\mathcal{H}}$, respectively.
The canonical projection $\mathcal{H} \to \underline{\mathcal{H}}$ gives a group homomorphism 
$$i : G \longrightarrow \underline{G}.$$
Let $R$ denote the abelian group generated by isomorphism classes of short exact sequences in $\mathcal{H}$.
Similarly, let $\underline{R}$ denote the abelian group generated by isomorphism classes of triangles in $\underline{\mathcal{H}}$.
By \cite[Lemma 4.3, with $A = \mathbb{k}$]{Qi}, triangles in $\underline{\mathcal{H}}$ are precisely those which are isomorphic to short exact sequences in $\mathcal{H}$, in other words, there is a surjection 
$$h : R \longrightarrow \underline{R}.$$
Let $f : R \to G$ be the group homomorphism that maps isomorphism classes of short exact sequences $[0 \to X \to Y \to Z \to 0]$ to $[X] - [Y] + [Z]$.
Similarly, let $k : \underline{R} \to \underline{G}$ be the group homomorphism that maps isomorphism classes of triangles $[X \to Y \to Z \to T X]$ to $[X] - [Y] + [Z]$:
\[\xymatrix{
R \ar[r]^f \ar[d]^h & G \ar[d]^i \\
\underline{R} \ar[r]^k & \underline{G} \\
}\]
Let $j : K \to \underline{K}$ be the cokernel of the above square in the horizontal direction:
\[\xymatrix{
R \ar[r]^f \ar[d]^h & G \ar[r]^g \ar[d]^i & K \ar[d]^j \ar[r] & 0\\
\underline{R} \ar[r]^k & \underline{G} \ar[r]^l & \underline{K} \ar[r] & 0 \\
}\]
The map $i : G \to \underline{G}$ is surjective by construction, this implies that $j$ is surjective.

Now replace $R$ and $\underline{R}$ with their respective images under $f$ and $k$, denoted $R'$ and $\underline{R}'$.
Then $i|_{R'} : R' \to \underline{R}'$ is surjective because $h$ was.
This gives two short exact sequences with a surjections as the vertical maps:
\[\xymatrix{
0 \ar[r] & R' \ar[r] \ar[d]^{i|_{R'}} & G \ar[r]^g \ar[d]^i & K \ar[d]^j \ar[r] & 0\\
0 \ar[r] & \underline{R}' \ar[r] & \underline{G} \ar[r]^l & \underline{K} \ar[r] & 0 \\
}\]

Now if we introduce kernels of the vertical maps we get:
\[\xymatrix{
& 0 \ar[d] & 0 \ar[d] & 0 \ar[d] & \\
0 \ar[r] & S \ar[r] \ar[d] & F \ar[r]^{g|_{F}} \ar[d] & I \ar[d] \ar[r] & 0 \\
0 \ar[r] & R' \ar[r] \ar[d]^{i|_{R'}} & G \ar[r]^g \ar[d]^i & K \ar[d]^j \ar[r] & 0\\
0 \ar[r] & \underline{R}' \ar[r] \ar[d] & \underline{G} \ar[r]^l \ar[d] & \underline{K} \ar[r] \ar[d] & 0 \\
& 0 & 0 & 0 & \\
}\]
Since the two bottom rows and all columns are exact, it follows from the nine lemma that also the top row is exact.
In particular, this tells us that if $y \in \operatorname{ker}(j) = I$, then $y = g(x)$ for some $x \in \mathrm{ker}(i) = F$.
But then $x = [A] - [B]$ for some objects $A$ and $B$ of $\mathcal{H}$.
The objects $A$ and $B$ are isomorphic in $\underline{\mathcal{H}}$ because $i([A]) = i([B])$.
This means that there are morphisms
\[\xymatrix{
A \ar@<2pt>[r]^s & \ar@<2pt>[l]^t B 
}\]
in the category $\mathcal{H}$ such that $1_A - ts$ factors trough a projective (or injective) object $P$:
\[\xymatrix@R=22pt{
A \ar[rd]_{1_A - ts} \ar[r]^u & P \ar[d]^v \\
& A \\
}\]
This means that $\begin{smallmatrix}( t & v )\end{smallmatrix} \left( \begin{smallmatrix} s \\ u \end{smallmatrix} \right) = 1_A$:
\[\xymatrix{
A \ar[rd]_{1_A} \ar[r]^(.4){\left( \begin{smallmatrix} s \\ u \end{smallmatrix} \right)} & B \oplus P \ar[d]^{\left( \begin{smallmatrix} t & v \end{smallmatrix} \right)} \\
& A \\
}\]
The morphism $\left( \begin{smallmatrix} s \\ u \end{smallmatrix} \right)$ is monic, we denote its cokernel by $Q$.
We then have the following short exact sequence in $\mathcal{H}$:
which is split because $\left( \begin{smallmatrix} s \\ u \end{smallmatrix} \right)$ had a left inverse.
This means that $B \oplus P$ is isomorphic to $A \oplus Q$ in $\mathcal{H}$.
On the other hand, the short exact sequence descends to a triangle in $\underline{\mathcal{H}}$ and since $\left( \begin{smallmatrix} s \\ u \end{smallmatrix} \right)$ represents an isomorphism there, $Q$ is a zero object in $\underline{\mathcal{H}}$. Hence $1_Q$ factors trough a projective, which implies that $Q$ is projective.
In conclusion, we have shown that $[B \oplus P] = [A \oplus Q]$ in $G$ and $x = [A] - [B] = [P] - [Q]$ in $F$ with $P$ and $Q$ projective. This tells us that $I$, as an abelian group, is generated by projective objects.

It remains to prove that $I$ is an ideal.
In the non-graded case, this follows from the fact that projective objects in $\mathcal{H}$ are summands of (finite-dimensional) free modules, tensor product distributes over direct sum, a tensor product with a free module is always free (see \cite[Proposition 2.1]{Qi}) and consequently, tensor products with projectives are projective.
In the graded case, this follows from the fact that being projective is reflected when forgetting the grading.
\end{proof}

\begin{proposition}\label{verdier_quotient}
Let $\mathcal{T}$ be a triangulated category and $\mathcal{S}$ a thick triangulated subcategory.
Then we have an exact sequence of Grothendieck groups
$$K_0(\mathcal{S}) \longrightarrow K_0(\mathcal{T}) \longrightarrow K_0\left(\mathcal{T}\big/\mathcal{S}\right) \longrightarrow 0$$
where $\mathcal{T} \big/ \mathcal{S}$ denotes Verdier localization (see \cite[Subsection 4.6]{Kr} for a definition) and the first two homomorphisms come from the corresponding inclusion and quotient functors.
If, in addition, $\mathcal{T}$ is a monoidal triangulated category and $\mathcal{S}$ an ideal, then the image of $K_0(\mathcal{S})$ is equal to the ideal given by the kernel of the morphism of Grothendieck rings $K_0(\mathcal{T}) \to K_0\left(\mathcal{T}\big/\mathcal{S}\right)$.
\end{proposition}

\begin{proof}
The categories $\mathcal{T}$ and $\mathcal{T}/\mathcal{S}$ have the same objects and the quotient functor $\mathcal{T} \to \mathcal{T}/\mathcal{S}$ maps triangles to triangles.
This means the groups $K_0(\mathcal{T})$ and $K_0(\mathcal{T}/\mathcal{S})$ have presentations given by the same generators, while the latter group is subject to strictly more relations.
Therefore, it is enough to show that every relation $[X] - [Y] + [Z] = 0$, coming from a triangle $X \to Y \to Z \to \Sigma X$ in $\mathcal{T}/\mathcal{S}$, can be recovered from the relations in $\mathcal{T}$ together with the generators in $\mathcal{S}$ (i.e.\ relations $[S] = 0$ for objects $S$ in $\mathcal{S}$).

Consider an arbitrary triangle in $\mathcal{T}/\mathcal{S}$
\[\xymatrix{
A \ar[r]^a & B \ar[r]^b & C \ar[r]^c & \Sigma A\\
}\]
and choose a representative $s^{-1}f$ of $a$, written in the calculus of left fractions of $\mathcal{T}/\mathcal{S}$.
We extend $f : A \to D$ into a triangle $A \to D \to E \to \Sigma A$ in $\mathcal{T}$ and fill in with a morphism $C \to E$ so that we get an morphism of triangles in $\mathcal{T}/\mathcal{S}$
\[\xymatrix{
A \ar[r]^{s^{-1}f} \ar@{=}[d] & B \ar[r] \ar[d]^s & C \ar[r] \ar[d] & \Sigma A \ar@{=}[d]\\
A \ar[r]^{f} & D \ar[r] & E \ar[r] & \Sigma A \\
}\]
which is, in fact, an isomorphism of triangles.
This shows that any triangle in $\mathcal{T}/\mathcal{S}$ is isomorphic to one coming from $\mathcal{T}$.
Thus it remains to show that objects which are isomorphic in $\mathcal{T}/\mathcal{S}$ are related under our restricted set of relations.

Write $X \leadsto Y$ if there is a triangle $X \to Y \to Z \to \Sigma X$ in $\mathcal{T}$ such that $Z$ lies inside of $\mathcal{S}$.
In particular we have $B \leadsto D$ in the digram above.
Now assume the morphism $s^{-1}f$ is an isomorphism.
Then $C$ is a zero object of $\mathcal{T}/\mathcal{S}$ and so is $E$ by being isomorphic to $C$.
Since $\mathcal{S}$ is thick, it follows (see \cite[Proposition 4.6.2]{Kr}) that $E$ belongs to $\mathcal{S}$ and we have $A \leadsto D$.
In conclusion, the relations $[E] = 0$ and $[A] - [D] + [E] = 0$ imply $[A] = [D]$ and analogously we have $[B] = [D]$, so we get $[A] = [B]$.

The statement about rings follows immediately.
\end{proof}
\pagebreak

\section{The Ring of Cyclotomic Integers}

Let $[n]_q$ denote the polynomial $(q^n - 1)\big/(q - 1) = 1 + q + q^2 + \cdots + q^{n-1}$ and let $\Phi_n(q)$ denote the $n$:th cyclotomic polynomial.
Note that $\Phi_n(q) = [n]_q$ when $n$ is prime.
A polynomial inside a parenthesis will usually denote an ideal, especially when occurring as a denominator with a ring as the numerator.
Whenever a polynomial occurs in a denominator without a parenthesis, we claim it divides the numerator.
In particular $[nm]_q\big/[n]_q = 1 + q^n + q^{2n} + \cdots + q^{(m - 1)n}$.

\begin{proposition}
The Grothendieck ring of the stable category $H_n\!-\!\underline{\mathrm{zmod}}$ (a.k.a.\ the homotopy category of $n$-complexes over vector spaces) is isomorphic to 
$$\Z[x]\big/\big([n]_x\big).$$
\end{proposition}

\begin{proof}
$K_0(H_n\!-\!\mathrm{zmod}) \simeq \Z[x, x^{-1}]$.
Then $K_0(H_n\!-\!\underline{\mathrm{zmod}}) \simeq \Z[x, x^{-1}]\big/([n]_x)$ by Proposition \ref{grothendieck_quotient}.
But 
$$\Z[x, x^{-1}]\big/\big([n]_x\big) \simeq \Z[x]\big/\big([n]_x\big)$$
because $0 = (x - 1)[n]_x = x^n - 1$, hence $x^{-1} = x^{n-1}$.
\end{proof}

\begin{proposition}\label{grothendieck_zmod}
The Grothendieck ring of the stable category $H_n \! \otimes \! H_m \!-\!\underline{\mathrm{zmod}}$ is isomorphic to
$$\Z[x, x^{-1}, y, y^{-1}]\big/\big([n]_x[m]_y\big).$$
\end{proposition}

\begin{proof}
The projective objects correspond to $x^a y^b [n]_x[m]_y$.
The statement follows from Proposition \ref{grothendieck_quotient}.
\end{proof}
Unlike the previous case, the variable $x^{-1}$ cannot be expressed in terms of $x$, $y$ and $y^{-1}$ (and analogously for the variable $y^{-1}$).
If that where possible, then any ring homomorphism from $\Z[x, y, y^{-1}]\big/([n]_x[m]_y)$ would map $x$ to an invertible element.
But the ring homomorphism to $\Z\big/(m)$, where $x \mapsto 0$, $y \mapsto 1$, provides a counterexample.

\begin{theorem}[Main result]
Assume $n$ and $m$ are distinct odd prime numbers.
Then we have an isomorphism of rings
$$K_0\left(\mathcal{T}\big/\mathcal{S}\right) \simeq \Z[q]\big/\big(\Phi_{nm}(q)\big),$$
where $\mathcal{T} = H_n \! \otimes \! H_m \!-\!\underline{\mathrm{zmod}}$ (see \ref{definition_zmod}), $\mathcal{S}$ is the thick ideal given in Corollary \ref{thick_ideal}, $\mathcal{T} \big/ \mathcal{S}$ is a Verdier localization (see \cite[Subsection 4.6]{Kr} for a definition) and $K_0$ denotes Grothendieck ring.
\end{theorem}

\begin{proof}
We prove this in four steps.

{\bf The first step} shows that
$$K_0\left(\mathcal{T}\big/\mathcal{S}\right) \simeq \frac{\Z[x, x^{-1}, y, y^{-1}]\big/([n]_x[m]_y)}{([n]_x, [m]_y)}$$
where $([n]_x, [m]_y)$ denotes the ideal generated by $[n]_x$ and $[m]_y$.
By Proposition \ref{grothendieck_zmod}, we have $K_0(\mathcal{T}) = \Z[x, x^{-1}, y, y^{-1}]\big/([n]_x[m]_y)$.
Now $\mathcal{S}$ is the direct sum of two ideals, each of which has an image under the inclusion into $K_0(\mathcal{T})$ given by $([n]_x)$ or $([m]_y)$.
That is $\operatorname{im}(K_0(\mathcal{S}) \to K_0(\mathcal{T}))$ is equal to the ideal $([n]_x, [m]_y)$.
The statement then follows from Lemma \ref{verdier_quotient}.

{\bf The second step} simplifies the expression by writing
$$\frac{\Z[x, x^{-1}, y, y^{-1}]\big/([n]_x[m]_y)}{([n]_x, [m]_y)} \simeq
\Z[x, y]\big/([n]_x, [m]_y).$$
This is because the principal ideal $([n]_x[m]_y)$ is contained in the ideal $([n]_x, [m]_y)$.
The variables $x^{-1}$ and $y^{-1}$ can be omitted because $x^n = y^m = 1$.

{\bf The third step} establishes an isomorphism of rings
$$\Z[x, y]\big/([n]_x, [m]_y) \simeq
\Z[q]\big/([nm]_q/[m]_q, [nm]_q/[n]_q).$$
To prove this, let $C_i$ denote the cyclic group with $i$ elements.
Choose an isomorphism of groups $C_n \times C_m \simeq C_{nm}$.
This provides an isomorphism of group rings:
$$\Z[C_n \times C_m] \simeq \Z[C_{nm}].$$
This isomorphism can be expressed as follows:
$$\Z[x, y]\big/(x^n - 1, y^m - 1) \simeq
\Z[q]\big/(q^{nm} - 1)$$
where the variables $x$, $y$ and $q$ correspond to generators of the groups $C_n$, $C_m$ and $C_{nm}$ respectively.
Now under the above isomorphism, the polynomial $[n]_x$ corresponds to $[nm]_q/[m]_q$ and analogously, $[m]_y$ corresponds to $[nm]_q/[n]_q$.
Thus we have an isomorphism
$$\frac{\Z[x, y]\big/(x^n - 1, y^m - 1)}{([n]_x, [m]_y)} \simeq
\frac{\Z[q]\big/(q^{nm} - 1)}{\left([nm]_q/[m]_q, [nm]_q/[n]_q\right)}.$$
And since the polynomials $[n]_x$, $[m]_y$ and $[nm]_q$ are divisors of $x^n - 1$, $y^m - 1$ and $q^{nm} - 1$ respectively, the latter relations become redundant.
This proves the third step.

{\bf The last step} shows that our ideal is principal, that is an isomorphism
$$\Z[q]\big/\left([nm]_q/[m]_q, [nm]_q/[n]_q\right) \simeq \Z[q]\big/(\Phi_{nm}(q)).$$
Since $n$ and $m$ where distinct primes, we have $[nm]_q = [n]_q[m]_q\Phi_{nm}(q)$ which shows that the ideal $(\Phi_{nm}(q))$ contains both the polynomial $[nm]_q/[m]_q$ and $[nm]_q/[n]_q$.
To show the other direction, choose positive integers $a$ and $b$ such that 
\begin{align*}
1 &= an - bm,\\
1 &= [an]_q - q[bm]_q,\\
1 &= \frac{[an]_q}{[n]_q}[n]_q - q\frac{[bm]_q}{[m]_q}[m]_q,\\
\Phi_{nm}(q) &= \frac{[an]_q}{[n]_q} \frac{[nm]_q}{[m]_q} - q\frac{[bm]_q}{[m]_q} \frac{[nm]_q}{[n]_q}.
\end{align*}
\end{proof}

\section*{Acknowledgements}
I am deeply grateful to my advisor Volodymyr Mazorchuk for invaluable support.
I am also thankful to Martin Herschend who taught me about Frobenius categories.

\end{document}